\documentclass[12pt, leqno]{amsart}
\usepackage{amssymb, amsmath}
\usepackage{latexsym}
\usepackage{graphics}
\usepackage{epsfig}
\usepackage{graphicx}
\usepackage{epstopdf}
\usepackage{color}
\usepackage{url}
\usepackage{amsthm}
\usepackage{cite}
\usepackage{fancyhdr}

\newtheorem{theorem}[subsection]{Theorem}
\newtheorem{proposition}[subsection]{Proposition}
\newtheorem{remark}[subsection]{Remark}
\newtheorem{lemma}[subsection]{Lemma}
\newtheorem{corollary}[subsection]{Corollary}

\newtheorem{definition}[subsection]{Definition}
\newtheorem{example}[subsection]{Example}
\newtheorem{assumption}[subsection]{Assumption}

\newcommand{\R}{\mathbb R}
\newcommand{\Q}{\mathbb Q}

\newcommand{\Z}{\mathbb Z}
\newcommand{\F}{\mathbb F}
\newcommand{\N}{\mathbb N}

\newcommand\cf[1]{\left< #1\right>}

\title[A new attack to RSA]{A new  attack to RSA with small private exponent  and partial information}
\author{Jorge Jim\'enez Urroz}

\begin{document}
\maketitle
\begin{abstract} We give a new algorithm to attack RSA with small private exponent, when some partial information of $p+q$ is given. The algorithm is a very simple modification of the original Wiener's attack with continued fractions, and allows us to factor $n$ whenever $d<n^{\frac{1+\delta}4}$ if we know  a $\delta$-fraction of the most significant bits of $p+q$. The algorithm is unconditional, which is not the case in previous  improvements that use Coppersmith method. As an example, 
our algorithm can be applied to break any cryptosystem with modulus $n$ of $512$ bits and  $d<n^{0.3}$, given an improvement in the original attack of Wiener. 
\end{abstract}

\section{Introduction}

Nowadays the organization of society is based on digital communication  and over $6000$ million people are connected. The efficiency and security needed for   encryption or digital signature on most of the protocols like SSH, PGP, OpenPGP, SSL/TLS, software, browsers etc,   rely on RSA and since it was created in $1978$ many papers have appeared trying to break it in very different scenarios.

\

To build RSA we need a public key an a secret key. The public key is made by two integers $n,e$, where  $n$ is the product of two unknown primes $p,q$,  and $e$ is coprime to $\varphi(n)=(p-1)(q-1)$ the Euler totient function. The secret key consist in the two unknown primes $p,q$ and the integer $d$ inverse of $e$ modulo $\varphi(n)$. Most of the attacks to RSA are based on this fact and use the identity 
\begin{equation}\label{eq:main}
ed-k\varphi(n)=1,
\end{equation} 
valid for some integer $k$, to get information on the secret exponent $d$ and the factors of $n$. 

\

When the difference in computer power between sender and receiver is high, we need to make the private and public exponents $d,e$ small, to make feasible computation in the lighter device. For example in a communication between a smart card and a computer, the smart card needs a short $d$ to decrypt, while the computer needs to have  a short $e$, so the smart card can encrypt. Nowadays is known that the smaller are those exponents, the more insecure is the protocol. A simple example comes with the following scenario: suppose we take  the public exponent to be   $e=3$ and send  the message  to three different people with the same exponent. Then we would have $c_1=m^3\pmod {n_1}$, $c_2=m^3\pmod {n_2}$ and $c_3=m^3\pmod {n_3}$ and then we just have to apply  the chinese remainder theorem to get $m^3\equiv c \pmod {n_1n_2n_3}$ but $m^3<n_1n_2n_3$ so $m^3=c$ and $m=c^{1/3}$.

\

But there is a lot more to study on small exponents. Almost $40$ years ago Wiener  proved in \cite{wiener}, using continued fractions in a clever way on Equation (\ref{eq:main}),  that small private exponents $d<cn^{1/4}$ for some constant $c$   are  insecure, and the attacker can recover the secret key in polynomial time. That paper was seminal and since then, until today, there is a strong interest in discovering attacks to break RSA for exponents $d$ of bigger size. We can find several papers trying to improve even  the constant $c$, up to the current record of  $c=\frac{1}{{18}^{1/4}}$,  given by  Susilo et al. \cite{susi} in $2019$. More information on the constant can be found in \cite{blemay, bodu,converse}. The authors of \cite{susi} also gave an example of $d\le \frac12n^{1/4}+1$ in which the method  does not break the RSA. This restricts  the value of the constant to $c\in(\frac1{2.06},\frac12]$ but also 
proves  that the exponent $1/4$ is the best one can hope, using the method as it was proposed by Wiener.

\

There are other variants of Wiener's attack, \cite{buto,duje,duje2,veti} that allow to break the RSA cryptosystem even if $d$ is few bits longer than the previous bound. The  references  \cite{duje,duje2,veti} are based on an identity of the form $d=rb_{m+1}+sb_m$ where $b_m,b_{m+1}$ are denominators of the convergents of $e/n$, while the  first,  \cite{buto}, uses another integer $n'$ derived from $n$ and construct the convergents of $e/n'$. Concretely considering the integer $n'=\left[n-\left(1+\frac3{2\sqrt 2}\right)n^{1/2}+1\right]$ they get up to $d^2e<8n^{3/2}$. Observe that for $e\sim n$ this increases Wiener's bound up to $d<2\sqrt 2n^{1/4}$. Very recently Trishing proved in \cite{trishin} that a variant of the attack also is able to break the RSA for large exponents $d$, and gives a proportion of $n^{-1/2}$ of weak exponents. 
 
 \
  
 Then, the paper by Boneh-Durfee \cite{bodu} comes, and strikes the community showing an improvement in the exponent, and gets up to $d<n^{0.292}$, at least in theory. Their method has changed completely from Wiener's and creates another school to attack RSA with small private exponent. Their tool is the powerful theorem of Coppersmith \cite{cop} to find small solutions to polynomial equations. Namely they consider $\varphi(n)=A+y$ for $A=(n+1)/2$ known, and $y=-(p+q)/2<n^{1/2}$ unknown, in Equation (\ref{eq:main}), and apply Coppersmith method to 
 $$
 f(x,y)=x(A+y)-1\equiv 0\pmod e,
 $$
 getting a solution for $|x|<n^{\delta}$ where $\delta=1-\frac{\sqrt2}{2}\approx 0.292.$
 
 \
 
Their result is asymptotic and to reach the previous value the dimension of the lattice needed makes the attack infeasible. In practice, they get only to exponents of order $0.28$ for moduli of $1000$ bits, while $0.255$ for moduli of $10000$ bits. There are other improvements of the method,  like  \cite{hemay, fenipa, lizhenqi,taku}, getting  better results in practice, but not  reaching  $d<n^{0.292}$ which remains the state of the art nowadays despite all the efforts.

 \

 But there are  other common ways of attackng  RSA, assuming that the attacker could have an extra information on the secret key. This information  normally comes from certain side channel attacks,  in the presence of convert channels, \cite{partial,timing}. The first of such instances was in \cite{bdf} where the authors, also with lattices,  recover the secret key knowing  only $\log d/4$ of the least significant bits of $d$. In \cite{steinzhen} the authors prove  that the running time  in \cite{bdf} depends exponentially on the least significant bits, LSBS for short, shared by $p$ and $q$. Other works as \cite{sama}, guessing some of the most significant bits,  MSBs for short,  of one factor of $n$ reduce the information needed on $d$, while \cite{ernst} improves the bounds by some refinement of the lattices. Other works using Coppersmith method and partial information on the secret key can be found in \cite{taku2,pehuhu,zhefenipa}. In \cite{zhefenipa} the authors go a bit further and  use Coppersmith method somehow combined with continued fractions,  to get a general exponent depending on the information we know about $e$, $d$ and the MSBs of $p+q$. Concretely if $e<N^\alpha$ and we know a proportion of $1/2-\gamma$ of the MSBs of $p+q$ one can recover any key $d<1-\alpha/3-\gamma/2$.  A generalization of the method to equations of the form $ed-1\equiv 0\pmod{(p^a-1)(q^a-1))}$ appeared very recently in \cite{general}.
 
 \
 
 On the other hand, using only continued fractions,  De Weger’s \cite{deweger} uses the most significant bits, MSBS for short, to prove that for primes with $|p-q|<n^{\beta}$, Wiener's attack will go up to $d<n^{3/4-\beta}$ and \cite{blomay} combine Wiener's attack with Coppersmith's original theorem to factor any modulus such that there is a linear relation of the form $ex+y\equiv 0\pmod \varphi(n)$ with $x$ and $y$ small. The authors in \cite{egypt} knowing and approximation of $p$ of order  $n^\alpha/8$ MSB's get up to $d<n^{(1-\alpha)/2}$ on Wiener's attack, by using convergents of different integers as in \cite{buto}.

 \
 
 

 
 \

It is of fundamental importance to remark now that even though the results using Coppersmith method are stronger, all of them depend on a hypothesis which works in practice, but never has been proved. Concretely the algebraic independence of the solutions given by the reduction algorithm in the lattice. On the other hand, results based on continued fractions are unconditional. 

\

{\bf Our contribution.}  The paper is dedicated to attack RSA, find the private key and factor the modulus, when extra information on the factors  is leaked. We will only use continued fractions and hence all the results are unconditional. In our main theorem we prove that  knowing an approximation $S$ of $p+q$, so that $p+q-S<n^{(1-\delta)/2}$, is enough to get the private exponent whenever $d<n^{\frac{1+\delta}4}$.  

\

We dedicate a section  to prove that an approximation of  $p+q$ is equivalent to knowing an approximation of  $p$ and $q$ separately, and it is in fact another way of expressing the knowledge of the MSBs of the corresponding integer. In this way we can  recover the results  in \cite{egypt}, in terms of approximations of $p+q$. We get the same result as they when $e\approx n$, but removing also the constant $1/8$ from that paper, and generalize it to any range of $e$. Also,  we need  to emphasize that  our method is extremely simple in comparison with them, which makes computation much faster than \cite{egypt}, since we only consider the computation of the convergents of a precise  fraction. 

\

The  estimation in our main result  is worse than in   \cite{zhefenipa}. However, we should mention that their results, apart from being conditional to finding algebraically independent solutions on the lattice reduction,  also use an extra hypothesis which seems unlikely to occur in certain range of the parameters that they consider. Concretely their results are based on the size of  the inverse of a particular integer $q_r$ modulo $q_{r-1}$, which normally will be of the same size as the modulus.

\

We should mention that the authors themselves mention that they find  a gap between their experiments and their theoretical results, without explaining why.



\

In Section \ref{cf} we include the necesary preliminaries on continued fractions. Then in Section \ref{results} we state and prove the results. Section \ref{approx} is dedicated to analyze the relation between approximation of an integer and thier MSBs, for $p+q$ and $p$ and $q$. Finally we include some experiments in Section \ref{experiments}. Throughout the paper, given two integers $a,b$ $[a,b]$ will denote thier least common multiple while $(a,b)$ will be their greatest common divisor.

\section{Preliminaries}\label{cf}

\medskip

To attack RSA with small $d$ we will use continued fractions. 
\begin{definition} Given a real number $\alpha$ the continued fraction of $\alpha$ is an expression of the form
\[
\alpha=
q_0+\cfrac{1}{q_1+\cfrac{1}{q_2+\cfrac{1}{q_3+\cfrac{1}{q_4+
  \cfrac{1}{q_5+\cfrac{1}{q_6+\cfrac{1}{q_7+\cfrac{1}{q_8+
  \cdots\vphantom{\cfrac{1}{1}} }}}}}}}}
\]
where $q_0=\left[\alpha\right]$, $r_0=\left\{\alpha\right\}$, $q_i=\left[1/r_{i-1}\right]$,  $r_i=\left\{1/r_{i-1}\right\}$ for $i\ge 1$, until $r_m=0$, which will occur exactly when $\alpha\in\Q$.  In this case the continued fraction is finite and we denote it as $\alpha=\cf{q_0,q_1,\dots,q_m}$
\end{definition}
\begin{remark} Note that $q_m\ge 2$ since otherwise $r_{m-1}=1$ which is impossible. 
\end{remark} 
  \
 
We can construct the continued fraction in the following way.  Let 
$\frac{n_i}{d_i}=\cf{q_0,\dots,q_{i}}$. Then one can prove \cite[Theorem 1]{kin} that 
\begin{equation}\label{eq:nd}
\begin{array}{ l l}
n_0=q_0, & d_0=1\\
n_1=q_0q_1+1,& d_1=q_1\\
n_i=q_in_{i-1}+n_{i-2}, &d_i=q_id_{i-1}+d_{i-2},\quad \text{for any }  2\le i\le m,\\
\end{array}
\end{equation}
and  \cite[Theorem 2]{kin}
\begin{equation}
n_id_{i-1}-d_{i}n_{i-1}=(-1)^{i+1}  \quad \text{for any }  1\le i\le m.
\end{equation}
This in particular implies 
\begin{equation}\label{eq:coprim}
(n_i,d_i)=1, 
\end{equation}
for all $0\le i\le m$.

\

We will also need to control the growth of the convergents. Concretely we have
\begin{theorem}\label{th:grow} Let $\alpha=\cf{q_0,\dots,q_{m}}$ and let  $\frac{n_i}{d_i}$ its convergents for $i=0,\dots, m$. Then for any $i\ge 4$,
$$
n_i>2^{\frac{i-1}{2}}.
$$
\end{theorem}
\begin{proof}  We will prove it by induction. We know that for $i\ge 1$,  $q_i\ge 1$,  and   $n_i\ge 1$.  Hence, noting that $n_3\ge n_2+n_1\ge 2$, we get  the first two cases  $n_4\ge n_3+n_2\ge 3>2^{3/2}$ and $n_5\ge n_4+n_3\ge 5>2^{4/2}$.
Now, suppose $n_{i-2}>2^{\frac{i-3}{2}}.$ Then, by (\ref{eq:nd}) for $i\ge 6$ 
$$
n_{i}=q_in_{i-1}+n_{i-2}\ge n_{i-1}+n_{i-2}\ge  n_{i-2}+n_{i-3}+n_{i-2}> 2n_{i-2}>2^{(i-1)/2}.
$$
\end{proof}
\begin{remark} See  \cite[Theorem 12]{kin} where the inequality is proved for the denominators of the convergents. We did not find a reference including the previous inequality for the numerators.
\end{remark}

\

The key fact on Wiener's method to attack RSA comes with the following lemma. 
\begin{lemma}[Wiener]\label{lem:wiener}  Let $\alpha=\cf{q_0,\dots,q_m}$, and  $\alpha'\in\Q$ such that $\alpha=\alpha'(1-\delta)$ for $\delta<\frac{2}{3n_md_m}$, as in (\ref{eq:nd}). Then  one of the convergents of $\alpha$ is $\alpha'$.
\end{lemma}
\begin{proof} See \cite[p. 555, (22)]{wiener}.
\end{proof}

We end this section with a simple lemma on the Euler function. We include the proof for convenience.
\begin{lemma}\label{lem:fi} Let $n=pq$ with $q<p<2q$. Then 
\begin{equation}\label{lem:fi}
n+1-\frac{3}{\sqrt 2}\sqrt n<\varphi(n)<n+1-2\sqrt n
\end{equation}
\end{lemma}
\begin{proof}
The function $f(p)=p+\frac np$ is increasing for $p>\sqrt n$, so under the hypothesis we get the lower bound $\sqrt n\le p$ and,  multiplying by $p$ in $p<2q$, the upper bound $p^2<2n$, so $2\sqrt n\le f(p)\le \frac{3}{\sqrt 2}\sqrt n$. The result follows by noting that $\varphi(n)=n+1-f(p)$.
\end{proof}

\section{Wiener's approach revisited} \label{results}

\begin{theorem} Let $n$ be an RSA modulus $n=pq$, with prime factors $q<p<2q$, and suppose  $p+q-1=x+R$ where $x$ is known. Let $e$ be  an
 integer coprime with $\varphi(n)$ and $d=e^{-1}\pmod {\varphi(n)}$, verifying (\ref{eq:main}). Then, if 
$d<(\sqrt{2/3}n-\sqrt {3n})(eR)^{-1/2}$ we can factor $n$ in polynomial  time in $\log (en)$.
\end{theorem}
\begin{proof} Dividing by $d(n-x)$ in (\ref{eq:main}) we get  
\begin{eqnarray}
\nonumber\frac{e}{n-x}&=&\frac{1}{d(n-x)}+\frac{k}{d}\frac{(p-1)(q-1)}{(n-x)}=\frac{k}{d}\left(\frac{n-p-q+1+\frac 1k}{n-x}\right)\\
&=&\frac{k}{d}\left(1-\frac{R-\frac{1}{k}}{n-x}\right).
\end{eqnarray}

 Hence, by Lemma \ref{lem:wiener} if $\frac{R-\frac{1}{k}}{n-x}<\frac{1}{3/2kd}$, or equivalently
 \begin{equation}\label{eq:bound}
kd<\frac23\frac{n-x}{R-1/k},
 \end{equation} 
 one of the convergents of $\frac{e}{n-x}$ is $\frac{k}{d}$. Again Equation (\ref{eq:main}) gives trivially $ed^2=kd\varphi(n)+d>kd\varphi(n)$, or $kd<\frac{ed^2}{\varphi(n)}$, so in particular if we have $d<(\sqrt{\frac 23}n-\sqrt {3n})(eR)^{-1/2}$, then 
 $ d<\sqrt{2/3}\varphi(n)(eR)^{-1/2}$ by Lemma \ref{lem:fi}, 
  and we get
 $$
 kd<\frac{ed^2}{\varphi(n)}<\frac23\frac{\varphi(n)}R<\frac23\frac{n-x}{R-1/k},
 $$
 and hence one of the convergents of $\frac{e}{n-x}$, say $\frac{n_i}{d_i}$ is $\frac{k}{d}$. Noting that $(n_i,d_i)=(k,d)=1$ by (\ref{eq:coprim}) and (\ref{eq:main}), we get $k=n_i$ and $d=d_i$. Then, we compute $\varphi(n)$ by (\ref{eq:main}), and from there the factors $p,q$.
 
 \
 
 By Theorem \ref{th:grow}, $e=n_m>2^{(m-1)/2}$, so we need at most $O(\log e)$ steps to find $\frac kd$, and  on each step  we perform one multiplication and one addition by \ref{eq:nd} with integers bounded by $e,n-x$, since the sequence of numerators and denominators are increasing with limit $e$, and $n-x$ respectively. 
\end{proof}

\begin{equation}\label{eq:alg}
\text{{\bf Algorithm}}
\end{equation}

\medskip

Input: $e,n,x$

Output: $d,p,q$

\medskip

for $i=1$ until $m$  we:

 \hskip10pt Compute $n_i$, $d_i$ the convergents of $\frac{e}{n-x}$.
 
 \hskip10pt  Compute $L=\frac{ed_i-1}{n_i}$. This is our candidate of $\varphi(n)$
 
 \hskip10pt Solve $X^2-(n-L)X+n=0$ to get the roots $x_1,x_2$ 
 
 \hskip10pt If $x_1, x_2\in\N$ and $x_1x_2=n$ 
 
 \hskip20pt	return ($d_i,x_1,x_2$)

\begin{remark} We should emphisize that the extra information needed on $p+q$ would be unnecessary if the size of the modulus is not too  big.  In particular since the algorithm runs in $O(\log en)$, one could repeat it $O(\sqrt n/2^l)$ times with a guess of $x=2^lA$ for  $2\sqrt n/2^l\le A<3/\sqrt 2\sqrt n/2^l$ and increasing $A$ by $1$  each time. This would give an algorithm running on $O(\sqrt n/2^l \log (en))$, which still is feasible for $l$  large enough. For example if $n$ is a modulus of $512$-bits, then $l=200$ would mean that we repeat the algorithm  $2^{56}$ times, and we capture $56$ bits of $p+q$ which is more than $1/10$ of the bits of $n$, and we could factor $n$ whenever $d<0.3$ with no previous information at all on the primes factors of $n$.

\end{remark}
\section{$\delta$-approximations}\label{approx}

In this section we establish, as we mentioned in the introduction, that knowing an approximation of  $p$ and $q$ is equivalent to knowing an approximation  of $p+q$.  Before proving it,   we will make a concrete definition for convenience in the exposition. Along the whole section we will denote all the constants by $c$, or $C$, but on each step of the argument could have different values.

\begin{definition} Let  $m$ be an  unknown integer and $\delta>0$. We say that the integers $K,U$ are a $\delta$-approximation of $m$ if $m=K+U$, $K$ is known and $|U|<cm^{1-\delta}$ for some constant $c$.
\end{definition}

We include two straighforward  lemmas for convenience.
\begin{lemma} Having a $\delta$-approximation of an integer $m$ is equivalent to  knowning the MSBs of $m$. 
\end{lemma}
\begin{proof}Concretely,  if we know the $L$ MSBs of an integer $m$, then we can write $m=2^lA+B$, where $l=[\log_2m]-L$, $A$ is a known integer  and $0\le B<2^l<m^{1-\frac{L}{\log_2m}}$, so we have a $\frac{L}{\log_2m}$-approximation of $m$. On the other hand, if $m=K+U$ with $|U|<cm^{1-\delta}<2^l$, then we have the inequalities $-1+\frac{K}{2^l}<\frac{m}{2^l}<\frac{K}{2^l}+1$, so 
$$
\left[\frac{m}{2^l}\right]=\left[\frac{K}{2^l}\right]-\varepsilon,
$$
where $\varepsilon=1$ if $U<0$, while $\varepsilon=0$ if $U>0$. Hence, 
$$
m=\sum_{j=0}^{[\log_2m]}m_j2^j=2^l\left[\frac{m}{2^l}\right]+\sum_{j=0}^{l}m_j2^j=2^l\left(\left[\frac{K}{2^l}\right]-\varepsilon\right)+B,
$$
where $0\le B<2^l$, and we know the $[\log_2m]-l$ MSBs of $m$.
\end{proof}
\begin{lemma}\label{lem:either}  Let the public integer $n=pq$  be the product of two  unknown prime numbers  $q<p<2q$ large enough. Having a  $\delta$-approximation of either $p$ or $q$ is equivalent to having a $\delta$-approximation of both $p$ and $q$.
\end{lemma}
\begin{proof} Obviously we just have to prove one implicaiton. Now,  note that $n=pq<\min\{p^2, 2q^2\}$. Then If $p=K_p+U_p$ for $|U_p|<Cp^{1-\delta}$, then $|U_p|<cp$, for some $c<1$ and $p$ large enough so $K_p>(1-c)p$ and 
$$
\left|\frac{n}{K_p}-\frac{n}{p}\right|=n\frac{|p-K_p|}{pK_p}<\frac{C}{1-c}qp^{-\delta}<\frac{C}{1-c}q^{1-\delta},
$$
so $q=K_q+ U_q$ where $K_q=\frac nK_p$ and $|U_q|<\frac{C}{1-c} q^{1-\delta}$. On the opposite, If we start  knowing $q=K_q+U_q$ for $|U_q|<Cq^{1-\delta}$, the reasoning is analogous.

\end{proof}
We now can prove our main result of apprximations.
\begin{proposition}\label{kbits}  Let the public integer $n=pq$  be the product of two  unknown prime numbers  $q<p<2q$.  Suppose $p+q=K+U$, is a $\delta$-approximation of $p+q$  with $K\ge \kappa \sqrt n$ for some $\kappa>2$. Then one can find a $\delta$-approximation of $p$ and a $\delta$-approximation of $q$ in polynomial time, and viceversa.
\end{proposition} 

\begin{proof}
 Suppose we know a $\delta$-approximation of each  $p$ and  $q$,  say $p=K_p+U_p$ for $|U_p|<c_p p^{1-\delta}$ and $q=K_q+U_q$ for $|U_q|<c_q q^{1-\delta}$, and denote $C=\max\{c_p,c_q\}$. Then we have 
$p+q=K_{p+q}+U_{p+q}$ where  $K_{p+q}=K_p+K_q$ and 
$$
|U_{p+q}|=|U_p+U_q|\le |U_p|+|U_q|<C(p^{1-\delta}+q^{1-\delta})<c\left(p+q\right)^{1-\delta},
$$ 
for some constant $c$,  using  $p<2(p+q)/3$, $q<(p+q)/2$ by the conditions on the primes. 

\

Now suppose $p+q=K_{p+q}+U_{p+q}$ where $K_{p+q}$ is a known integer with $K_{p+q}>\kappa\sqrt n$, for some constant $\kappa>2$  and  $|U_{p+q}|<C(p+q)^{1-\delta}$ for some constant $C$. By Lemma \ref{lem:either} it is enough to have a $\delta$-approximation of one of the primes, say $p$. Then solving the equation $x^2-(p+q)x+n=0$ with solutions $p,q$, we get 
\begin{eqnarray*}
&&p=\frac{K_{p+q}+U_{p+q}+\sqrt{(K_{p+q}+U_{p+q})^2-4n}}{2}\\
&&=K_p+U_p
\end{eqnarray*}
where 
$$
K_p=\frac{K_{p+q}}{2}+\frac{\sqrt{K_{p+q}^2-4n}}{2},
$$
and 
$$
U_p=\frac{U_{p+q}^2+2K_{p+q}U_{p+q}}{2\left(\sqrt{(K_{p+q}+U_{p+q})^2-4n}+\sqrt{K_{p+q}^2-4n}\right)}+\frac{U_{p+q}}2
$$
Now, by the conditions on the primes, $C_1p\le K_{p+q}<C_2p$,  for some constants $C_1,C_2$, , so 
$U_{p+q}^2<|K_{p+q}U_{p+q}|<Cp^{2-\delta}$, for some constant $C$,  while 
\begin{eqnarray*}
&&\sqrt{(K_{p+q}+U_{p+q})^2-4n}+\sqrt{K_{p+q}^2-4n}>\sqrt{K_{p+q}^2-4n}\\
&&>(\sqrt{c^2-4})n^{1/2}>Cp,
\end{eqnarray*}
for some constant $C$, so 
$$
U_p<C(p^{1-\delta}+(p+q)^{1-\delta})<cp^{1-\delta}.
$$
for some constants $C,c$. The $\delta$-approximation then will be $p=K+U$ where 
$K=\left[K_p\right]$ and $U=K_p-K+U_p.$
\end{proof}
\begin{remark} If the $\delta$-approximation of $p+q$ is with any $K$ with no lower bound, then the same reasoging gives a $\delta/2$-approximation of $p$ and $q$. In this case we should use 

\begin{eqnarray*}
&&p=\frac{K_{p+q}+U_{p+q}+\sqrt{(K_{p+q}+U_{p+q})^2-4n}}{2}\\
&&=K_p+U_p
\end{eqnarray*}
where 
$$
K_p=\frac{K_{p+q}}{2}+\frac{\sqrt{K_{p+q}^2-4n}}{2},
$$
and 
\begin{eqnarray*}
U_p&=&\frac{U_{p+q}^2+2K_{p+q}U_{p+q}}{2\left(\sqrt{(K_{p+q}+U_{p+q})^2-4n}+\sqrt{K_{p+q}^2-4n}\right)}+\frac{U_{p+q}}2<\\
&&<\sqrt{U_{p+q}^2+2K_{p+q}U_{p+q}}<cp^{1-\delta/2}.
\end{eqnarray*}
for some constant $c$.

\end{remark}
\section{Experiments}\label{experiments}
All the experiments have been done in python using a personal computer, concretely a Mac book pro with 2GHz intel core i5 of $4$ kernels and $16$GB of memory. The program used is included in the Appendix \ref{program}
\begin{example} Let $n=123543221$, and $e=3782461$. $n^{1/4}=105.4277$, so assume $p+q=22130+R$, where $R<106$. So we consider the continued fraction of $\frac{e}{n-x}=\frac{3782461}{123521091}$ to get as its list of convergents
\begin{eqnarray*}
&&[(0, 1), (1, 32), (1, 33), (2, 65), (3, 98), (32, 1045), \\
&&(1187, 38763), (4780, 156097), (10747, 350957), \\
&&(15527, 507054), (57328, 1872119)]
\end{eqnarray*}
When reaching the $5$-th convergent, multiply by $e$ and take the integer part, we get $123520992$ and quotient $1$. Letting 
$$
123543221-123520992=22229=a,
$$ 
and solving the equation
$$
x^2-22230x+n,
$$
we get $p=11117,q=11113$. In this case $d>n^{0.373}$.
\end{example}
\begin{example} The next   example is  used by the authors of \cite{zhefenipa}. They need more than $9$ seconds to factor $n$. This algorithm takes miliseconds.
\begin{eqnarray*}
n &=&4974472773881384118014647795711547516478259956645613937
\\
&&6718994215049702937974938606611020238731850953950803408
\\
&&79242549924017891702063138282911199989251977\\
e&=&3362135212646195373732698994059691822491400319644492284\\
&&6372466250218118962134587236033368546938732470674860097\\
&&55286359769772827397995978362545037878464433\\
x&=&14279176227683147138044792728299966749821\times 10^{37}
\end{eqnarray*}

\vskip10pt

returns

\begin{eqnarray*}
&&4387254887864516616725764744278199052823488,\\
&& 6491196400927340728938155227779304732403793,\\
&&86
\end{eqnarray*} 
whose second coordinate is the secret exponent on the $86$ convergent, in miliseconds.     In this case the exponent is $d\approx n^{0.2785}$. 

\

Taking
\begin{eqnarray*}
e&=&4403407016802090047393873499278161087517856177983207067\\
&&7248297745308257030345671722133330755179411640118454952769\\
&&35054951436828450021261154418374318704225
\end{eqnarray*}
the program returns
\begin{eqnarray*}
&&574601190489207276358664128996126904791408049924536747334,\\ &&649119640092734072893815522777930473240379394857234598713,\\
&&110    
\end{eqnarray*}
whose second coordinate is the secret exponent on the $110$ convergent, also miliseconds.     In this case the exponent is $d\approx n^{0.369}$. 
\end{example}
\begin{example} This is also an example  used by the authors of \cite{zhefenipa}. While they need $17$ digits of the sum of the two primes, to get the exponent $d\approx n^{0.317}$, we get the same with $21$ digits. Observe that the $4$ digits of difference can be obtained by repeating the algorithm $10^4$ times.

Now take 
\begin{eqnarray*}
n&=&667881289802154038823628680614836895352147070157135359\\
&&517370786146585569663805767598976469948361102331675600\\
&&6175305149520572078488495000186844811399641203\\
e&=&481939619752981969109709347529972401122215764606900174\\
&&834517626107402697849846501416569215357100143887963073\\
&&4388003452670588809411338901610567571522175099\\
x&=&10^{25}y \text{ where }\\
y&=&16637133402080618303507029737134206559438091398368285
\end{eqnarray*}
\end{example}
\section{Conclusion} 

The main objective of the paper is to establish  an unconditional improvement on the size of the secret exponent which is vulnerable to Wiener's attack with continued fractions, knowing partial information on $p+q$. We also prove that knowing this information is equivalent to knowing  partial information on $p$ and $q$. We observe that up to $60$ bits, this information could be found by repeating the algorithm until succeed, which would happen in polynomial time. In particular  if $n$ is a modulus of $512$-bits, then running the algorithm   $2^{56}$ times , with consecutive brute force guesses of $x$, we could factor $n$ whenever $d<0.3$.  

\section{Appendix} \label{program}
The program first compute the gcd of two integer numbers, collecting the quotients on each step with the following instructions in python:

{\bf Computing the gcd} 

def mcd(a,b):

    a,b=max(a,b),min(a,b)

    q=[]

    while b!=0:

        q.append(a//b)

        c=a\%b

        a=b

        b=c

    return q

Then, given the quotients, we compute the sequence of numerators and denominators of the convergents by the formulas in (\ref{eq:nd}). For each convergent, we try to factor $n$ by guessing $k,d$ as numerator and denominator of the convergent,  with the Algorithm \ref{eq:alg}, with the following instructions:
 
 \indent q=mcd(n-x,e)
 
 \indent   num,den=[q[0],q[0]*q[1]+1],[1,q[1]]
  
    for j in range(2,len(q)):
  
        num.append(q[j]*num[j-1]+num[j-2])
  
        den.append(q[j]*den[j-1]+den[j-2])    
  
        for l in range(len(num)):
         	
	\hskip 10pt $L=(ed_i-1)/n_i$.

	\hskip 10pt import sympy as sym
	
	x=sym.symbols('x')

eq = sym.Eq(x**2 - 1, 0)

roots = sym.solve(ecuacion, x)

	if  isinstance(roots[0], int):
  
                print(roots)

\bibliographystyle{plain}
\bibliography{small}

\end{document}